\documentclass[11pt]{article}
\usepackage{amssymb,amsmath}
\usepackage[mathscr]{eucal}
\usepackage[cm]{fullpage}
\usepackage[english]{babel}
\usepackage[latin1]{inputenc}
\usepackage{tikz}
\usepackage{amsthm}

\def\im{\mathop{\mathrm{Im}}\nolimits}
\def\id{\mathrm{id}}

\def\T{\mathscr{T}}
\def\PT{\mathscr{PT}}
\def\I{\mathscr{I}}
\def\Sym{\mathscr{S}}
\newcommand{\End}{\mathrm{End}}
\newcommand{\wEnd}{\mathrm{wEnd}}
\newcommand{\sEnd}{\mathrm{sEnd}}
\newcommand{\swEnd}{\mathrm{swEnd}}
\newcommand{\Aut}{\mathrm{Aut}}
\newcommand{\transf}[1]{\left(\begin{smallmatrix}#1\end{smallmatrix}\right)}
\newtheorem{theorem}{Theorem}[section]
\newtheorem{proposition}[theorem]{Proposition}
\newcommand{\lastpage}{\addresss}

\newcommand{\addresss}{\small {\sf

\noindent{\sc V\'\i tor H. Fernandes},
Center for Mathematics and Applications (NOVA Math)
and Department of Mathematics,
Faculdade de Ci\^encias e Tecnologia,
Universidade Nova de Lisboa,
Monte da Caparica,
2829-516 Caparica,
Portugal;
e-mail: vhf@fct.unl.pt.

\medskip

\noindent{\sc J\"{o}rg Koppitz},
Institute of Mathematics and Informatics,
Bulgarian Academy of Sciences,
1113 Sofia,
Bulgaria;
e-mail: koppitz@math.bas.bg.

\medskip 

\noindent{\sc Tiwadee Musunthia}, 
Department of Mathematics, 
Faculty of Science, Silpakorn University, 
Nakorn Pathom, Thailand 73000; 
e-mail: tiwadee.m@gmail.com. 
}}

\title{Presentations for monoids of endomorphisms of a star graph}

\author{V\'\i tor H. Fernandes\footnote{This work is funded by national funds through the FCT - Funda\c c\~ao para a Ci\^encia e a Tecnologia, I.P., under the scope of the (Center for Mathematics and Applications) projects UIDB/00297/2020 (https://doi.org/10.54499/UIDB/00297/2020) and UIDP/00297/2020 (https://doi.org/10.54499/UIDP/00297/2020).}, J\"org Koppitz~and T. Musunthia
}

\begin{document}

\maketitle

\begin{abstract}
In this paper, we consider the monoids of all endomorphisms, of all weak endomorphisms, 
of all strong endomorphisms and of all strong weak endomorphisms
of a star graph with a finite number of vertices.
Our main objective is to exhibit a presentation for each of them.
\end{abstract}

\medskip



\section*{Introduction}\label{presection}

For a set $\Omega$, denote by $\PT(\Omega)$ the monoid (under composition) of all
partial transformations on $\Omega$, by
$\T(\Omega)$ the submonoid of $\PT(\Omega)$ of all full transformations on $\Omega$ 
and by $\Sym(\Omega)$ the \textit{symmetric group} on $\Omega$,
i.e. the subgroup of $\PT(\Omega)$ of all permutations on $\Omega$.

\smallskip

Let $G=(V,E)$ be a simple graph (i.e. undirected, without loops and without multiple edges). 
Let $\alpha$ be a full transformation of $V$. We say that $\alpha$ is: 
\begin{itemize}
\item an \textit{endomorphism} of $G$ if $\{u,v\}\in E$ implies  $\{u\alpha,v\alpha\}\in E$, for all $u,v\in V$;
\item a \textit{weak endomorphism} of $G$ if $\{u,v\}\in E$ and $u\alpha\ne v\alpha$ imply  $\{u\alpha,v\alpha\}\in E$, for all $u,v\in V$;
\item a \textit{strong endomorphism} of $G$ if $\{u,v\}\in E$ if and only if  $\{u\alpha,v\alpha\}\in E$, for all $u,v\in V$;
\item a \textit{strong weak endomorphism} of $G$ if $\{u,v\}\in E$ and $u\alpha\ne v\alpha$ if and only if $\{u\alpha,v\alpha\}\in E$, for all $u,v\in V$;
\item an \textit{automorphism} of $G$ if $\alpha$ is a bijective strong endomorphism (i.e. $\alpha$ is bijective and $\alpha$ and $\alpha^{-1}$ are both endomorphisms). 
For finite graphs, any bijective week endomorphism is an automorphism. 
\end{itemize}

Denote by:
\begin{itemize}
\item $\End G$ the set of all endomorphisms of $G$;
\item $\wEnd G$ the set of all weak endomorphisms of $G$;
\item $\sEnd G$ the set of all strong endomorphisms of $G$;
\item $\swEnd G$ the set of all strong weak endomorphisms of $G$;
\item $\Aut G$ the set of all automorphisms of $G$. 
\end{itemize}

Clearly, $\End G$, $\wEnd G$, $\sEnd G$, $\swEnd G$ and $\Aut G$ are monoids under composition of mappings with the identity mapping $\id$ as the identity element. Moreover, $\Aut G$ is also a group.  
It is also clear that $\Aut G\subseteq\sEnd G\subseteq\End G,\swEnd G\subseteq\wEnd G$ 
\begin{center}
\begin{tikzpicture}[scale=0.5]
\draw (0,0) node{$\bullet$} (0,1) node{$\bullet$} (-1,2) node{$\bullet$} (1,2) node{$\bullet$} (0,3) node{$\bullet$}; 
\draw (0.9,0.1) node{\scriptsize$\Aut G$} (1.1,1.1) node{\scriptsize$\sEnd G$} (-2.3,2.1) node{\scriptsize$\swEnd G$} (1.95,2.1) node{\scriptsize$\End G$} (1.2,3.1) node{\scriptsize$\wEnd G$};
\draw[thick] (0,0) -- (0,1); 
\draw[thick] (0,1) -- (-1,2); 
\draw[thick] (0,1) -- (1,2); 
\draw[thick] (-1,2) -- (0,3);
\draw[thick] (1,2) -- (0,3);
\end{tikzpicture}
\end{center}
(these inclusions may not be strict).

\smallskip

Monoids of endomorphisms of graphs have many important applications, particularly related to automata theory; see \cite{Kelarev:2003}.
A large number of interesting results concerning graphs and algebraic properties of their endomorphism monoids have been obtained by several authors (see, for example, \cite{Bottcher&Knauer:1992,Fan:1996,Gu&Hou:2016,Hou&Luo&Fan:2012,
%Kelarev&Praeger:2003,
Knauer:2011,Knauer&Wanichsombat:2014, Li:2003,Wilkeit:1996}).
In recent years, the first two authors together with Dimitrova and Quinteiro, 
also have studied such types of monoids, in particular by considering finite undirected paths and cycles;
see \cite{Dimitrova&Fernandes&Koppitz&Quinteiro:2020,Dimitrova&Fernandes&Koppitz&Quinteiro:2021,Dimitrova&Fernandes&Koppitz&Quinteiro:2023arxiv}.

\smallskip

For any non negative integer $n$,  let $\Omega_n=\{1,2,\ldots,n\}$ and $\Omega_n^0=\{0, 1,\ldots,n\}=\{0\}\cup\Omega_n$.
Notice that, $\Omega_0=\emptyset$ and $\Omega_0^0=\{0\}$.
For an integer $n\geqslant1$, consider the \textit{star graph}
$$
S_n=(\Omega_{n-1}^0, \{\{0,i\}\mid 1\leqslant i\leqslant n-1\})
$$
with $n$ vertices.
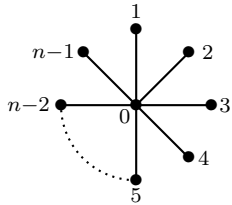
\begin{center}
\begin{tikzpicture}
\draw (0,0) node{$\bullet$} (0,2) node{$\bullet$} (-1,1) node{$\bullet$} (1,1) node{$\bullet$};
\draw (0.7,0.3) node{$\bullet$} (0.7,1.7) node{$\bullet$} (-0.7,1.7) node{$\bullet$};
\draw (0,-0.2) node{$\scriptstyle5$} (0,2.25) node{$\scriptstyle1$} (-1.43,1) node{$\scriptstyle n-2$} (1.18,1) node{$\scriptstyle3$};
\draw (0.9,0.3) node{$\scriptstyle4$} (0.95,1.7) node{$\scriptstyle2$} (-1.1,1.7) node{$\scriptstyle n-1$};
\draw (0,1) node{$\bullet$}; \draw (-.15,.85) node{$\scriptstyle0$};
\draw[thick] (0,1) -- (0,0); \draw[thick] (0,1) -- (0,2); \draw[thick] (0,1) -- (-1,1); \draw[thick] (0,1) -- (1,1);
\draw[thick] (0,1) -- (0.7,0.3); \draw[thick] (0,1) -- (0.7,1.7) ; \draw[thick] (0,1) -- (-0.7,1.7);
\draw[thick,dotted] (0,0) arc (-90:-180:1);
\end{tikzpicture}
\end{center}

These very basic graphs, which are special examples of \textit{trees} and also of \textit{complete bipartite graphs}, play a significant role in Graph Theory,
for instance, through the notions of \textit{star chromatic number} and \textit{star arboricity}.
We may also find important applications of star graphs in Computer Science.
For example, in Distributed Computing the \textit{star network} is one of the most common computer network topologies.

In 2023, Fernandes and Paulista \cite{Fernandes&Paulista:2023} considered the monoid of all partial isometries of the star graph $S_n$. 
They determined the rank (i.e. the minimum size of a generating set) and the cardinality of this monoid as well as described its Green's relations and exhibited a presentation.
On the other hand, recently, the first two authors together with Dimitrova considered in \cite{Dimitrova&Fernandes&Koppitz:2024,Dimitrova&Fernandes&Koppitz:2024sub}
the \textit{partial counterparts} of the monoids $\sEnd(S_n)$, $\swEnd(S_n)$, $\End(S_n)$ and 
$\wEnd(S_n)$, studying their regularity, describing their Green's relations and computing their cardinalities and ranks. 

\medskip 

In this paper, our main aim is to determine presentations for the monoids $\End S_n$, $\swEnd S_n$ and $\wEnd S_n$, which we will carry out in Section \ref{presentations}. 
In the first section, we will look at some basic concepts and results about presentations. 
In Section \ref{basics}, we describe the monoids $\Aut S_n$, $\sEnd S_n$, $\End S_n$, $\swEnd S_n$ and $\wEnd S_n$, and recall some of their unpublished properties, namely cardinalities, regularity, generators and ranks, which were studied by the first author together with his undergraduate students Duarte Grilo and Samuel Medalha within the \textit{Undergraduate Research Opportunity Program} of NOVA School of Science and Technology. For the sake of completeness, we present here brief proofs of these results. 

\medskip 

For general background on Semigroup Theory and standard notations, we would like to refer the reader to Howie's book \cite{Howie:1995}.
Regarding Algebraic Graph Theory, we refer to Knauer's book \cite{Knauer:2011}.
We would also like to point out that we made use of computational tools, namely GAP \cite{GAP4}.

\section{On presentations} 

Let $X$ be a set and denote by $X^*$ the free monoid generated by
$X$. Often, in this context, the set $X$ is called an \textit{alphabet} and its elements are called \textit{words}.
A \textit{monoid presentation} is an ordered pair $\langle X\mid
R\rangle$, where $X$ is an alphabet and $R$ is a subset of
$X^*\times X^*$. An element $(u,v)$ of $X^*\times X^*$ is called a
\textit{relation} of $X^*$ and it is usually represented by $u=v$.
A monoid $M$ is said to be \textit{defined by a presentation} $\langle X\mid R\rangle$ if $M$ is
isomorphic to $X^*/{\sim_R}$, where $\sim_R$ denotes the congruence on $X^*$ generated by $R$,
i.e. $\sim_R$ is the smallest
congruence on $X^*$ containing $R$.
Suppose that $X$ is a generating set of a monoid $M$ and let $u=v$ be a relation of $X^*$.
We say that a relation $u=v$ of $X^*$ is \textit{satisfied} by $X$ if $u=v$ is an equality in $M$.
For more details, see \cite{Lallement:1979} or \cite{Ruskuc:1995}.

A well-known direct method to obtain a presentation for a monoid, that we might consider folklore,
is given by the following result; see, for example, \cite[Proposition 1.2.3]{Ruskuc:1995}.

\begin{proposition}\label{provingpresentation}
Let $M$ be a monoid generated by a set $X$.
Then, $\langle X\mid R\rangle$ is a presentation for $M$ if and only
if the following two conditions are satisfied:
\begin{enumerate}
\item
The generating set $X$ of $M$ satisfies all the relations from $R$;
\item
If $w_1,w_2\in X^*$ are any two words such that
the generating set $X$ of $M$ satisfies the relation $w_1=w_2$, then $w_1\sim_R w_2$.
\end{enumerate}
\end{proposition}

For finite monoids, a standard method to find a presentation
is described by the following result, adapted for the monoid case from
\cite[Proposition 3.2.2]{Ruskuc:1995}.

\begin{theorem}[Guess and Prove method] \label{ruskuc}
Let $M$ be a finite monoid, let $X$ be a generating set for $M$,
let $R\subseteq X^*\times X^*$ be a set of relations and let
$W\subseteq X^*$. Assume that the following conditions are
satisfied:
\begin{enumerate}
\item The generating set $X$ of $M$ satisfies all the relations from $R$;
\item For each word $w\in X^*$, there exists a word $w'\in W$ such that $w\sim_R w'$;
\item $|W|\le|M|$.
\end{enumerate}
Then, $M$ is defined by the presentation $\langle X\mid R\rangle$.
\end{theorem}

Notice that, if $W$ satisfies the above conditions, then, in fact,
$|W|=|M|$.

\smallskip

Let $X$ be an alphabet, let $R\subseteq X^*\times X^*$ be a set of
relations and let $W$ be a subset of $X^*$. We say that $W$ is a set of
\textit{canonical forms} for a finite monoid $M$
if Conditions 2 and 3 of Theorem \ref{ruskuc} are satisfied. Suppose
that the empty word belongs to $W$ and, for each letter $x\in X$ and
for each word $w\in W$, there exists a word $w'\in W$ such that $wx\sim_R w'$.
Then, it is easy to show that $W$ satisfies Condition 2.

In this paper, we will mainly use the method given by Theorem \ref{ruskuc} and the condition described above instead of Condition 2.

\smallskip

A presentation for the symmetric group $\Sym(\Omega_n)$ was determined by Moore \cite{Moore:1897} by the end of the 19th century (1897).
Nearly six decades later (1958), A\u{\i}zen\v{s}tat \cite{Aizenstat:1958}  gave a presentation for
full transformation monoid $\T(\Omega_n)$.
A few years later (1961),
presentations for the partial transformation monoid $\PT(\Omega_n)$
and for the symmetric inverse monoid $\I(\Omega_n)$
were found by Popova \cite{Popova:1961}.
In 1962, A\u{\i}zen\v{s}tat \cite{Aizenstat:1962} and Popova \cite{Popova:1962} exhibited presentations for the monoids of
all order-preserving transformations and of all order-preserving partial transformations of a finite chain, respectively, and from the sixties until our days several authors obtained presentations for many classes of monoids.
Further examples can be found, for instance, in
\cite{
%Cicalo&al:2015,
East:2011,
%Feng&al:2019,
Fernandes:2001,
Fernandes:2002survey,
Fernandes&Gomes&Jesus:2004,
Fernandes&Quinteiro:2016,
Howie&Ruskuc:1995,
Ruskuc:1995}.

\medskip 

We finish this section by recalling the presentations of $\Sym(\Omega_n)$, $\T(\Omega_n)$ and $\PT(\Omega_n)$ aforementioned; see \cite{Fernandes:2002survey}. 

Let $n\geqslant 3$ and $a=\transf{1 & 2 & 3 & \cdots & n \\ 2 & 1 & 3 & \cdots & n}$, 
$b=\transf{1 & 2 & \cdots & n-1 & n \\ 2 & 3 & \cdots &n & 1}$, 
$c=\transf{2 & 3 & \cdots & n \\  2 & 3 & \cdots & n}$ and 
$ e=\transf{1 & 2 & 3 & \cdots & n \\ 1 & 1 & 3 & \cdots & n}$. 
Then, we have: 
for $n\geqslant3$, 
\begin{equation*}
\Sym(\Omega_n)=\langle a,b\mid a^2 = b^n = (ba)^{n-1} = (ab^{n-1}ab)^3 = (ab^{n-j}ab^j)^2 = 1 ~ (2\leqslant j\leqslant n-2) \rangle;  
\end{equation*}  
for $n\geqslant4$, 
\begin{equation*}
\begin{split} 
\T(\Omega_n)=
\langle a,b, e\mid a^2 = b^n = (ba)^{n-1} = (ab^{n-1}ab)^3 = (ab^{n-j}ab^j)^2 = 1 ~ (2\leqslant j\leqslant n-2), \\  
a e = b^{n-2}ab^2 eb^{n-2}ab^2 = bab^{n-1}ab eb^{n-1}abab^{n-1} = ( ebab^{n-1})^2 =  e,  \\
(b^{n-1}ab e)^2 =  eb^{n-1}ab e = ( eb^{n-1}ab)^2, ~ ( ebab^{n-2}ab)^2 = (bab^{n-2}ab e)^2\rangle 
\end{split}
\end{equation*} 
and 
\begin{equation*}
\begin{split} 
\PT(\Omega_n) = \langle a,b, c, e \mid 
 a^2 =b^n =(ba)^{n-1} =(ab^{n-1}ab)^3 =(ab^{n-j}ab^j)^2 =1 ~ (2\leqslant j\leqslant n-2), \\ 
b^{n-1}ab cb^{n-1}ab=ba cab^{n-1}= c= c^2, ~ ( ca)^2= ca c=(a c)^2, \\
 a  e = b^{n-2}ab^2  e b^{n-2}ab^2 = bab^{n-1}ab  e b^{n-1}abab^{n-1} = ( ebab^{n-1})^2 =  e, \\  
 (b^{n-1}ab  e)^2 =  e b^{n-1}ab  e = ( eb^{n-1}ab)^2, ~ ( e bab^{n-2}ab)^2=(bab^{n-2}ab  e)^2, \\ 
  e c= ca c,~  c e= ca,~  ea c= ea, ~ 
  e ab^{n-1}aba  c= ab^{n-1}aba  c ab^{n-1}aba  e ab^{n-1}aba 
\rangle; 
\end{split} 
\end{equation*} 
for $n=3$,  
\begin{equation*}
\begin{split} 
\T(\Omega_3)=
\langle a,b, e\mid a^2 = b^3 = (ba)^2 = (ab^2ab)^3 = 1, ~ 
a e = bab^2ab eb^2abab^2 = ( ebab^2)^2 =  e, ~   \\ 
(b^2ab e)^2 =  eb^2ab e = ( eb^2ab)^2\rangle 
\end{split}
\end{equation*} 
and 
\begin{equation*}
\begin{split} 
\PT(\Omega_3) = \langle a,b, c, e \mid 
 a^2 = b^3 = (ba)^2 = (ab^2ab)^3 = 1, ~ %\\ 
b^{2}ab cb^{2}ab=ba cab^{2}= c= c^2, ~ ( ca)^2= ca c=(a c)^2, \\
a e = bab^2ab eb^2abab^2 = ( ebab^2)^2 =  e, ~ (b^2ab e)^2 =  eb^2ab e = ( eb^2ab)^2, \\  
  e c= ca c,~  c e= ca,~  ea c= ea, ~ 
  e ab^{2}aba  c= ab^{2}aba  c ab^{2}aba  e ab^{2}aba 
\rangle.
\end{split} 
\end{equation*} 

\section{Descriptions, regularity and ranks} \label{basics}

Let $\PT^0_{n-1}=\{\alpha\in\T(\Omega_{n-1}^0)\mid\mbox{$0\alpha=0$}\}$, 
$\T^0_{n-1}=\{\alpha\in\T(\Omega_{n-1}^0)\mid\mbox{$0\alpha=0$ and $\Omega_{n-1}\alpha\subseteq\Omega_{n-1}$}\}$
and let $\Sym^0_{n-1}=\Sym(\Omega_{n-1}^0)\cap\T^0_{n-1}$. 
Then, it is clear that $\PT^0_{0}=\T^0_{0}=\T(\Omega_0^0)=\{\transf{0\\0}\}=\Sym(\Omega_0^0)=\Sym^0_0$ and, 
for $n\geqslant2$, 
$\PT^0_{n-1}$ is a submonoid of $\T(\Omega_{n-1}^0)$ isomorphic to $\PT(\Omega_{n-1})$, 
$\T^0_{n-1}$ is a submonoid of $\T(\Omega_{n-1}^0)$ isomorphic to $\T(\Omega_{n-1})$ 
and $\Sym^0_{n-1}$ is a subgroup of $\Sym(\Omega_{n-1}^0)$ isomorphic to $\Sym(\Omega_{n-1})$. 
Moreover, for $n\geqslant1$, it a routine matter to check that:
\begin{itemize}
\item $\sEnd S_n=\End S_n=\T^0_{n-1}\cup\left\{\transf{0&1&\cdots&n-1\\i&0&\cdots&0} \mid i\in\Omega_{n-1}\right\}$; 

\item $\swEnd S_n  =  \End S_n\cup \{\transf{0&1&\cdots&n-1\\i&i&\cdots&i}\mid i\in\Omega_{n-1}^0\}$;  

\item $\wEnd S_n =\{\alpha\in\T(\Omega_{n-1}^0)\mid \mbox{$0\alpha\ne0$ implies $\im\alpha\subseteq\{0,0\alpha\}$}\}$. 
\end{itemize}
Therefore, we have: 
\begin{itemize}
\item $|\End S_n|=(n-1)^{n-1}+n-1$, for $n\geqslant 1$; 

\item $|\swEnd S_n|=(n-1)^{n-1}+2n-1$, for $n\geqslant 2$; 

\item $|\wEnd S_n|=n^{n-1}+(n-1)2^{n-1}$, for $n\geqslant 1$. 
\end{itemize}

Observe that, $\Aut S_1 = \End S_1 = \swEnd S_1 = \wEnd S_1 = \T(\Omega_0^0)= \{\transf{0\\0}\}$, 
$\Aut S_2 = \End S_2 =\Sym(\Omega_1^0) = \{\transf{0&1\\0&1},\transf{0&1\\1&0} \}$  
and $\swEnd S_2 = \wEnd S_2 =\T(\Omega_1^0) = \{\transf{0&1\\0&1},\transf{0&1\\1&0}, \transf{0&1\\0&0},\transf{0&1\\1&1} \}$. 
Furthermore, 
as the automorphisms of a graph are its bijective endomorphisms, for $n\geqslant3$, we have 
\begin{itemize}
\item $\Aut S_n=\Sym^0_{n-1}$ and so $|\Aut S_n|=(n-1)!$. 
\end{itemize} 

\smallskip 

It is obvious that, for $n=1,2$, the monoids $\End S_n$, $\swEnd S_n$ and $\wEnd S_n$ are all regular. 
In fact, it is easy to show that this is the case for any $n\geqslant1$: 

\begin{theorem} 
For $n\geqslant1$,  the monoids $\End S_n$, $\swEnd S_n$ and $\wEnd S_n$ are all regular. 
\end{theorem}
\begin{proof}
By the previous observation, we may consider $n\geqslant3$. 
As $\T^0_{n-1}$ is a submonoid of $\T(\Omega_{n-1}^0)$ isomorphic to $\T(\Omega_{n-1})$ and it is well known that this last monoid is regular, 
each element of $\T^0_{n-1}$ is regular in $\T^0_{n-1}$ and so also regular in $\End S_n$, $\swEnd S_n$ and $\wEnd S_n$. 

Let $\alpha_i=\transf{0&1&\cdots&n-1\\i&0&\cdots&0}$ for $1\leqslant i\leqslant n-1$. 
Then, $\alpha_i^3=\alpha_i$, whence $\alpha_i$ is regular in $\End S_n$ and so $\alpha_i$ is also regular in $\swEnd S_n$ and $\wEnd S_n$, 
for all $1\leqslant i\leqslant n-1$. 
We can conclude right away that $\End S_n$ is a regular monoid.

Let $\varepsilon_i=\transf{0&1&\cdots&n-1\\i&i&\cdots&i}$ for $0\leqslant i\leqslant n-1$. Then, $\varepsilon_i$ is an idempotent, 
whence $\varepsilon_i$ is regular in $\swEnd S_n$ and so $\varepsilon_i$ is also regular in $\wEnd S_n$, 
for all $1\leqslant i\leqslant n-1$. At this point, we can conclude that $\swEnd S_n$ is a regular monoid.

Now, let $\alpha\in\wEnd S_n$. 

First, suppose that $0\alpha=0$ and take $\beta\in \T(\Omega_{n-1}^0)$ such that $\alpha=\alpha\beta\alpha$. 
Notice that such $\beta$ exists since $\T(\Omega_{n-1}^0)$ is a regular monoid. 
Define $\bar\beta\in\T(\Omega_{n-1}^0)$ by $0\bar\beta=0$ and $i\bar\beta=i\beta$ for $i\in\Omega_{n-1}$. 
Then, $\bar\beta\in\wEnd S_n$ and $\alpha=\alpha\bar\beta\alpha$, 
whence $\alpha$ is regular in $\wEnd S_n$. 

Next, suppose that $0\alpha\neq0$. If $\im\alpha=\{0\alpha\}$, then $\alpha$ is an idempotent and so $\alpha$ is a regular element. 
Hence, suppose that $\im\alpha=\{0,0\alpha\}$. Take $i\in0\alpha^{-1}$ and let $\beta=\transf{0&1&\cdots&n-1\\i&0&\cdots&0}\in\wEnd S_n$. 
Therefore, $\alpha=\alpha\beta\alpha$ and so $\alpha$ is regular in $\wEnd S_n$, as required. 
\end{proof} 

\smallskip 

From now on and throughout this paper we will always consider $n\geqslant 3$. 
Observe that, in this case, we have $\Aut S_n\subsetneq \End S_n\subsetneq \swEnd S_n\subsetneq \wEnd S_n$. 

\smallskip 

Let us define
$$
a_0=\begin{pmatrix}
0&1 & 2 & 3 & \cdots & n-1 \\
0&2 & 1 & 3 &\cdots & n-1
\end{pmatrix}, ~ 
b_0=\begin{pmatrix}
0&1 & 2 & \cdots & n-2 &n-1 \\
0&2 & 3 & \cdots & n-1 & 1
\end{pmatrix}
~\text{and}~
e_0=\begin{pmatrix}
0&1 & 2 & 3 & \cdots & n -1\\
0&1 & 1 & 3 & \cdots & n-1
\end{pmatrix}. 
$$
Then,  $\{a_0,b_0,e_0\}$ is a generating set with minimum size of the monoid $\T^0_{n-1}$ (notice that, $a_0=b_0$ for $n=3$); see \cite{Fernandes:2002survey}. 
Let us also consider the following transformations of $\T(\Omega_{n-1}^0)$:
$$
c_0=\begin{pmatrix}
0 & 1 & 2 & \cdots & n-1 \\
0 & 0 & 2 &\cdots & n-1
\end{pmatrix},~
z=\begin{pmatrix}
0 & 1 & 2 & \cdots & n -1\\
1 & 0 & 0 & \cdots & 0
\end{pmatrix}
~\text{and}~
z_0=\begin{pmatrix}
0 & 1 & 2 & \cdots & n -1\\
0 & 0 & 0 & \cdots & 0
\end{pmatrix}.
$$
Notice that,  $\{a_0,b_0,e_0,c_0\}$ is a generating set with minimum size of the monoid $\PT^0_{n-1}$; see \cite{Fernandes:2002survey}. 
Moreover, it is easy to check that $a_0,b_0,e_0,c_0,z,z_0\in\wEnd(S_n)$. Furthermore, we have: 

\begin{theorem} 
For $n\geqslant4$, $\{ a_0,b_0,e_0,z \}$, $\{ a_0,b_0,e_0,z,z_0 \}$ and $\{ a_0,b_0,e_0,c_0,z \}$ are generating sets with minimum size of 
$\End S_n$, $\swEnd S_n$ and $\wEnd S_n$, respectively. Consequently, $\End S_n$, $\swEnd S_n$ and $\wEnd S_n$ have ranks $4$, $5$ and $5$, respectively.  
\end{theorem}
\begin{proof}
Since $\{a_0,b_0,e_0\}$ is a generating set of the monoid $\T^0_{n-1}$ and, for all $i\in\Omega_{n-1}$, 
$\transf{0&1&\cdots&n-1\\i&0&\cdots&0}=z\transf{0&1&\cdots&n-1\\0&i&\cdots&i}$ and 
$\transf{0&1&\cdots&n-1\\i&i&\cdots&i}=z_0\transf{0&1&\cdots&n-1\\i&0&\cdots&0}$,  
it immediately follows that $\{ a_0,b_0,e_0,z \}$ and $\{ a_0,b_0,e_0,z,z_0 \}$ are generating sets of $\End S_n$ and $\swEnd S_n$, respectively. 

Next, observe that $\wEnd S_n = \PT_{n-1}^0\cup \{\alpha\in\T(\Omega_{n-1}^0)\mid \mbox{$0\alpha\ne0$ and $\im\alpha\subseteq\{0,0\alpha\}$}\}$. 
Since $\PT_{n-1}^0$ is generated by $\{a_0,b_0,e_0,c_0\}$, 
in order to show that $\{a_0,b_0,e_0,c_0,z\}$ generates $\wEnd S_n$, 
it suffices to show that 
$\{\alpha\in\T(\Omega_{n-1}^0)\mid \mbox{$0\alpha\ne0$ and $\im\alpha\subseteq\{0,0\alpha\}$}\}$ is contained in the monoid generated by $\{a_0,b_0,e_0,c_0,z\}$. 
So, take $\alpha\in\T(\Omega_{n-1}^0)$ such that $0\alpha\ne0$ and $\im\alpha\subseteq\{0,0\alpha\}$. Then, 
$\alpha=\transf{0&A&B\\0\alpha& 0\alpha & 0}$, 
for some (disjoint) subsets $A$ and $B$ of $\Omega_{n-1}$ such that $A\cup B=\Omega_{n-1}$. 
Let $\alpha_0=\transf{0&A&B\\0& 0 & 0\alpha}$ and $\beta=\transf{0&1&\cdots&n-1\\0&0\alpha&\cdots&0\alpha}$. 
Then, $\alpha_0,\beta\in \PT_{n-1}^0$ and $\alpha=\alpha_0 z \beta$, which concludes the proof. 

Finally, it is easy to conclude that any generating set of $\End S_n$ cannot have less than four elements and that both 
$\swEnd S_n$ and $\wEnd S_n$ cannot be generated by less than five elements. 
\end{proof}

For $n=3$, as $b_0=a_0$ and $e_0=z^2$, it is a routine matter to conclude that 
$\{ a_0,z \}$, $\{ a_0,z,z_0 \}$ and $\{ a_0,c_0,z \}$ are generating sets with minimum size of $\End S_3$, $\swEnd S_3$ and $\wEnd S_3$, respectively.  
Therefore, $\End S_3$, $\swEnd S_3$ and $\wEnd S_3$ have ranks $2$, $3$ and $3$, respectively. 

\section{Presentations}\label{presentations} 

Recall that $\T^0_{n-1}$ and $\PT^0_{n-1}$ are monoids isomorphic to $\T(\Omega_{n-1})$ and $\PT(\Omega_{n-1})$, respectively. 

\smallskip 

Let $\langle a_0,b_0,e_0\mid R\rangle$ be a presentation of  $\T^0_{n-1}$ on the generators $a_0$, $b_0$ and $e_0$. 
Then, we have: 

\begin{theorem} \label{endpres}
For $n\geqslant 4$, 
$$
\langle a_0,b_0,e_0,z \mid R, a_0z = b_0z = e_0z = z,~ z^2=(e_0b_0)^{n-3}e_0 
\rangle
$$ 
is a presentation of $\End S_n$ on the generators $a_0$, $b_0$, $e_0$ and $z$. 
\end{theorem}
\begin{proof} 
First, notice that, it is a routine matter to show that the set of generators $\{a_0,b_0,e_0,z\}$ satisfies all the relations 
$R$, $b_0z = a_0z = e_0z = z$ and $z^2=(e_0b_0)^{n-3}e_0$. 

Denote by $\sim_R$ the congruence  on $\{a_0,b_0,e_0\}^*$ generated by $R$ and by $\sim$ the congruence on $\{a_0,b_0,e_0,z\}^*$ 
generated by $R\cup\{b_0z = a_0z = e_0z = z, z^2=(e_0b_0)^{n-3}e_0\}$.  

Let $w\in\{a_0,b_0,e_0,z\}^*$. Then $w =  u_1zu_2z\cdots u_kzu$, for some $u_1,u_2,\ldots,u_k,u\in\{a_0,b_0,e_0\}^*$ and $k\geqslant0$. 
As a consequence of the relations $b_0z = a_0z = e_0z = z$, we have $u_iz\sim z$, for $1\leqslant i\leqslant k$, and so 
$w\sim z^ku$. On the other hand, considering also the relation $z^2=(e_0b_0)^{n-3}e_0$, 
we also get $z^3 =  z^2z\sim (e_0b_0)^{n-3}e_0z\sim z$, 
whence $z^ku\sim zu$, if $k$ is odd,  and $z^ku\sim z^2u\sim (e_0b_0)^{n-3}e_0u$, if $k$ is an even integer greater than zero. 
Therefore, we obtain $w\sim zu'$ or $w\sim u'$, for some $u'\in\{a_0,b_0,e_0\}^*$. 

Now, let $w_1,w_2\in\{a_0,b_0,e_0,z\}^*$ be such that $w_1=w_2$ (equality in $\End S_n$). Our goal is to show that $w_1\sim w_2$, with which the proof is concluded. 

Let $u_1,u_2\in\{a_0,b_0,e_0\}^*$ be such that $w_i\sim u_i$ or $w_i\sim zu_i$, for $i=1,2$. Hence, we may consider four cases. 

\noindent{\sc case 1:} $w_1\sim u_1$ and $w_2\sim u_2$. 

Then, $u_1=w_1=w_2=u_2$ (equalities in $\End S_n$) and so $u_1=u_2$ (equality in $\mathscr{T}^0_{n-1}$). 
Since $\langle a_0,b_0,e_0\mid R\rangle$ is a presentation of  $\mathscr{T}^0_{n-1}$, we have $u_1\sim_R u_2$, 
whence $u_1\sim u_2$ and so $w_1\sim w_2$. 

\noindent{\sc case 2:} $w_1\sim u_1$ and $w_2\sim zu_2$. 
In this case, we have 
$$
\transf{0 & 1 & \cdots & n-1 \\ 0 & 1u_1 & \cdots & (n-1)u_1}=u_1=w_1=
w_2=zu_2=\transf{0 & 1 & \cdots & n-1 \\ 1 & 0 & \cdots & 0}\transf{0 & 1 & \cdots & n-1 \\ 0 & 1u_2 & \cdots & (n-1)u_2}= 
\transf{0 & 1 & \cdots & n-1 \\ 1u_2 & 0 & \cdots & 0}, 
$$
whence $u_1=\transf{0 & 1 & \cdots & n-1 \\ 0 & 0 & \cdots & 0}\not\in\mathscr{T}^0_{n-1}$,
which is a contradiction. Thus, this case does not occur. 

\noindent{\sc case 3:} $w_1\sim zu_1$ and $w_2\sim u_2$. 

Similar to the previous case, this case does not occur.

\noindent{\sc case 4:} $w_1\sim zu_1$ and $w_2\sim zu_2$. 

Then, $zu_1=w_1=w_2=zu_2$ (equalities in $\End S_n$), whence $zu_1=zu_2$ and so $z^2u_1=z^2u_2$. 
Since $z^2=(e_0b_0)^{n-3}e_0$, we get $(e_0b_0)^{n-3}e_0u_1=(e_0b_0)^{n-3}e_0u_2$, which is an equality in $\mathscr{T}^0_{n-1}$. 
Therefore, as $\langle a_0,b_0,e_0\mid R\rangle$ is a presentation of  $\mathscr{T}^0_{n-1}$, we have 
$(e_0b_0)^{n-3}e_0u_1\sim_R (e_0b_0)^{n-3}e_0u_2$, whence $(e_0b_0)^{n-3}e_0u_1\sim (e_0b_0)^{n-3}e_0u_2$ and so $z^2u_1\sim z^2u_2$, 
from which follows that $z^3u_1\sim z^3u_2$ and, as $z^3\sim z$, finally, that $zu_1\sim zu_2$. Thus $w_1\sim w_2$, by transitivity, as required. 
\end{proof}

\smallskip 

Next, let $\langle a_0,b_0,e_0,z\mid Q\rangle$ be a presentation of  $\End S_n$ on the generators $a_0$, $b_0$, $e_0$ and $z$. 
Hence, we get: 

\begin{theorem} For $n\geqslant 4$, 
$$
\langle a_0,b_0,e_0,z,z_0 \mid Q, 
a_0z_0=b_0z_0=e_0z_0=zz_0=z_0^2=z_0a_0=z_0b_0=z_0e_0=z_0 
\rangle
$$ 
is a presentation of $\swEnd S_n$ on the generators $a_0$, $b_0$, $e_0$, $z$ and $z_0$. 
\end{theorem}
\begin{proof} 
We begin by observing that we can routinely show that the set of generators $\{a_0,b_0,e_0,z,z_0\}$ satisfies all the relations 
$Q$ and  $a_0z_0=b_0z_0=e_0z_0=zz_0=z_0^2=z_0a_0=z_0b_0=z_0e_0=z_0$. 

Denote by $\sim_Q$ the congruence  on $\{a_0,b_0,e_0,z\}^*$ generated by $Q$ and by $\sim$ the congruence on $\{a_0,b_0,e_0,z,z_0\}^*$ 
generated by $Q\cup\{a_0z_0=b_0z_0=e_0z_0=zz_0=z_0^2=z_0a_0=z_0b_0=z_0e_0=z_0\}$.  

Let $w\in\{a_0,b_0,e_0,z,z_0\}^*\setminus\{a_0,b_0,e_0,z\}^*$. 
Then, $w =  u_1z_0u_2z_0\cdots u_kz_0u$, for some $u_1,u_2,\ldots,u_k,u\in\{a_0,b_0,e_0,z\}^*$ and $k\geqslant1$. 
As a consequence of the relations $a_0z_0=b_0z_0=e_0z_0=zz_0=z_0$, we get $u_iz_0\sim z_0$, for $1\leqslant i\leqslant k$, and so 
$w\sim z_0^ku\sim z_0u$, also in view of the relation $z_0^2=z_0$. 
On the other hand, we showed in the proof of Theorem \ref{endpres} that $u\sim_Q zu'$ or $u\sim_Q u'$, for some $u'\in\{a_0,b_0,e_0\}^*$. 
Hence, $w\sim z_0zu'$ or $w\sim z_0u'\sim z_0$, for some $u'\in\{a_0,b_0,e_0\}^*$, 
also taking into account the relations $z_0a_0=z_0b_0=z_0e_0=z_0$. 
Observe that, as transformation, we obtain 
$w=\transf{0 & 1 & \cdots & n-1 \\ i & i & \cdots & i}$, for some $i\in\Omega_{n-1}^0$, and so $w\not\in\End S_n$. 

Now, let $w_1,w_2\in\{a_0,b_0,e_0,z,z_0\}^*$ be such that $w_1=w_2$ (equality in $\swEnd S_n$). We aim to show that $w_1\sim w_2$. 
By the last observation above, we can conclude that $w_1,w_2\in\{a_0,b_0,e_0,z\}^*$ or $w_1,w_2\in\{a_0,b_0,e_0,z,z_0\}^*\setminus\{a_0,b_0,e_0,z\}^*$. 
If  $w_1,w_2\in\{a_0,b_0,e_0,z\}^*$, then $w_1\sim_Q w_2$ and so $w_1\sim w_2$. 
So, suppose that $w_1,w_2\in\{a_0,b_0,e_0,z,z_0\}^*\setminus\{a_0,b_0,e_0,z\}^*$.  
Hence, as transformation, 
$w_1=w_2=\transf{0 & 1 & \cdots & n-1 \\ i & i & \cdots & i}$, for some $i\in\Omega_{n-1}^0$. 
If $i=0$, then we must have $w_1\sim z_0\sim w_2$. On the other hand, suppose that $i\in\Omega_{n-1}$. 
Then, $w_1\sim z_0zu'_1$ and $w_2\sim z_0zu'_2$, for some $u'_1,u'_2\in\{a_0,b_0,e_0\}^*$, 
and (as transformations) $zu'_1=zu'_2=\transf{0 & 1 & \cdots & n-1 \\ i & 0 & \cdots & 0}$. 
Thus, $zu'_1\sim_Q zu'_2$ and so $w_1\sim z_0zu'_1\sim z_0zu'_2\sim w_2$, as required. 
\end{proof}

\smallskip 

Now, let $\langle a_0,b_0,e_0,c_0\mid S\rangle$ be a presentation of  $\PT^0_{n-1}$ on the generators $a_0$, $b_0$, $e_0$ and $c_0$. 
Thus, we obtain:  

\begin{theorem} For $n\geqslant 4$, 
\begin{align*} 
\langle a_0,b_0,e_0,c_0,z \mid S,~ 
a_0z = b_0z = e_0z = z,~ 
z^2=(e_0b_0)^{n-3}e_0,~ 
z^2c_0=zc_0
\rangle
\end{align*} 
is a presentation of $\wEnd S_n$ on the generators $a_0$, $b_0$, $e_0$, $c_0$ and $z$. 
\end{theorem}
\begin{proof} 
First, observe that it is a routine matter to check that the set of generators $\{a_0,b_0,e_0,c_0,z\}$ satisfies all the considered relations. 

Next, recall that $\wEnd S_n$ is the disjoint union of $\PT_{n-1}^0$ and 
$$
\left\{ \transf{0&A&B\\i& i & 0}\in\T(\Omega_{n-1}^0) \mid   A\cup B=\Omega_{n-1},~ A\cap B=\emptyset,~ i\in\Omega_{n-1} \right\}.
$$ 
For each $A,B\subseteq\Omega_{n-1}$ such that $A\cup B=\Omega_{n-1}$ and $A\cap B=\emptyset$, let $u_{A,B}$ be a word of 
$\{a_0,b_0,e_0,c_0\}^*$ representing the transformation $\transf{0&A&B\\0& 0 & 1}\in\PT_{n-1}^0$. 
For each $1\leqslant i\leqslant n-1$, let $u_i$ be a word of 
$\{a_0,b_0,e_0\}^*$ representing the transformation $\transf{0&1&\cdots&n-1\\0& i & \cdots &i}\in\T_{n-1}^0$. 
Notice that, as transformations, we have $\transf{0&A&B\\i& i & 0}=u_{A,B}zu_i$, 
for any disjoint subsets $A$ and $B$ of $\Omega_{n-1}$ such that $A\cup B=\Omega_{n-1}$ and $1\leqslant i\leqslant n-1$.  
Let $W_0$ be a set of canonical forms for the monoid $\PT^0_{n-1}$ regarding the set of generators $\{a_0,b_0,e_0,c_0\}$. 
Let 
$$
W_1=\{u_{A,B}zu_i\in \{a_0,b_0,e_0,c_0,z\}^*\mid A\cup B=\Omega_{n-1},~ A\cap B=\emptyset,~ 1\leqslant i\leqslant n-1\}
$$ 
and let $W=W_0\cup W_1$. Then, clearly, $|W|=|\wEnd S_n|$.  

Denote by $\sim_S$ the congruence  on $\{a_0,b_0,e_0,c_0\}^*$ generated by $S$ and by $\sim$ the congruence on $\{a_0,b_0,e_0,c_0,z\}^*$ 
generated by $S\cup\{a_0z = b_0z = e_0z = z,~ z^2=(e_0b_0)^{n-3}e_0,~ z^2c_0=zc_0\}$.  

Let $w\in W_0$. Since $\langle a_0,b_0,e_0,c_0\mid S\rangle$ is a presentation of  $\PT^0_{n-1}$, 
for each $x\in \{a_0,b_0,e_0,c_0\}$, we have $wx\sim_S w'$, whence $wx\sim w'$, for some $w'\in W_0$. 
On the other hand,  being $A=\{k\in\Omega_{n-1}\mid (k)w=0\}$ and $B=\Omega_{n-1}\setminus A$, 
we have $w(e_0b_0)^{n-3}e_0 = u_{A,B}$ and $(e_0b_0)^{n-3}e_0 = u_1$ (considering $w$, $u_1$ and $u_{A,B}$ as transformations of $\PT^0_{n-1}$), 
whence $w(e_0b_0)^{n-3}e_0 \sim_S u_{A,B}$ and $(e_0b_0)^{n-3}e_0 \sim_S u_1$, and so 
$$
wz\sim w (e_0b_0)^{n-3}e_0 (e_0b_0)^{n-3}e_0 z \sim u_{A,B} z^2z = u_{A,B} zz^2 \sim  u_{A,B} z (e_0b_0)^{n-3}e_0 
\sim  u_{A,B} z u_1 \in W_1. 
$$

Now, take a subset $A$ of $\Omega_{n-1}$, $B=\Omega_{n-1}\setminus A$ and $1\leqslant i\leqslant n-1$. Let $w=u_{A,B}zu_i\in W_1$. 
For each $x\in \{a_0,b_0,e_0\}$, we have $u_ix=u_{(i)x}$ (considering $u_i$ and $x$ as transformation of $\T_{n-1}^0$), whence $u_ix\sim_S u_{(i)x}$ and so 
$
wx=u_{A,B}zu_ix\sim u_{A,B}zu_{(i)x}\in W_1. 
$
On the other hand, (as transformations) we have  $u_1c_0=c_0u_1c_0$ and $u_ic_0=u_i$, if $2\leqslant i\leqslant n-1$, whence 
$$
wc_0=u_{A,B}zu_1c_0\sim  u_{A,B}zc_0u_1c_0 \sim  u_{A,B}z^2c_0u_1c_0 \sim  u_{A,B} (e_0b_0)^{n-3}e_0  c_0u_1c_0 \sim w', 
$$
for some $w'\in W_0$, and $wc_0=u_{A,B}zu_ic_0\sim u_{A,B}zu_i =w \in W_1$, if $2\leqslant i\leqslant n-1$. 
Finally, since $u_i \in \{a_0,b_0,e_0\}^*$, we have $u_iz\sim z$ and so $wz=u_{A,B}zu_iz \sim u_{A,B}z^2\sim u_{A,B}(e_0b_0)^{n-3}e_0 \sim w'$, 
for some $w'\in W_0$. 

Thus, we proved that $W$ is a set of canonical forms for $\wEnd S_n$, which completes this proof. 
\end{proof}

\medskip 

We finish this paper by exhibiting also presentations for $\End S_3$, $\swEnd S_3$ and $\wEnd S_3$. These presentations can be verified by using GAP \cite{GAP4}: 
\begin{equation*}
\begin{split} 
\End S_3 =
\langle a_0,z\mid a_0^2 =  1, ~ a_0z = z, ~z^3=z\rangle ;  
\end{split}
\end{equation*} 
\begin{equation*}
\begin{split} 
\swEnd S_3 =
\langle a_0,z,z_0\mid a_0^2 =  1, ~ a_0z = z, ~z^3=z, 
a_0z_0=zz_0=z_0^2=z_0a_0=z_0z^2=z_0\rangle ; 
\end{split}
\end{equation*} 
\begin{equation*}
\begin{split} 
\wEnd S_3 =
\langle a_0,c_0,z\mid a_0^2 =  1, ~ a_0z = z, ~z^3=z, 
c_0^2 = c_0,~  (c_0a_0)^2 = (a_0c_0)^2 = c_0a_0c_0, ~ \\ 
z^2c_0 = c_0a_0c_0,~  c_0z^2 = c_0a_0, ~ z^2a_0c_0 = z^2a_0,~ 
z^2c_0 = zc_0 
\rangle . 
\end{split}
\end{equation*}

\bigskip

\lastpage


\begin{thebibliography}{00}

\bibitem{Aizenstat:1958}
A.Ya. A\u{\i}zen\v{s}tat,
Defining relations of finite symmetric semigroups,
Mat. Sb. N. S.  45 (1958), 261--280 (Russian).

\bibitem{Aizenstat:1962}
A.Ya. A\u{\i}zen\v{s}tat,
The defining relations of the endomorphism semigroup of a finite linearly ordered set,
Sibirsk. Mat. 3 (1962), 161--169 (Russian).

\bibitem{Bottcher&Knauer:1992}
M. B\"{o}ttcher and U. Knauer,
\textit{Endomorphism spectra of graphs},
Discrete Mathematics 109 (1992), 45--57.

\bibitem{Dimitrova&Fernandes&Koppitz:2024}
I. Dimitrova, V. H. Fernandes and J. Koppitz,
\textit{On partial endomorphisms of a star graph},
Quaestiones Mathematicae (2024), 1--21
 (https://doi.org/10.2989/16073606.2024.2374796).
 
 \bibitem{Dimitrova&Fernandes&Koppitz:2024sub}
I. Dimitrova, V. H. Fernandes and J. Koppitz,
\textit{Presentations for monoids of partial endomorphisms of a star graph},
arXiv:2410.21084 [math.RA]  
(https://doi.org/10.48550/arXiv.2410.21084). 

\bibitem{Dimitrova&Fernandes&Koppitz&Quinteiro:2020}
I. Dimitrova, V. H. Fernandes, J. Koppitz and T. M. Quinteiro,
\textit{Ranks of monoids of endomorphisms of a finite undirected path},
Bull. Malaysian Math. Sci. Soc. 43(2) (2020), 1623--1645.

\bibitem{Dimitrova&Fernandes&Koppitz&Quinteiro:2021}
I. Dimitrova, V. H. Fernandes, J. Koppitz and T. M. Quinteiro,
\textit{Partial Automorphisms and Injective Partial Endomorphisms of a Finite Undirected Path},
Semigroup Forum 103(1), (2021), 87--105.

\bibitem{Dimitrova&Fernandes&Koppitz&Quinteiro:2023arxiv}
I. Dimitrova, V. H. Fernandes, J. Koppitz and T. M. Quinteiro,
\textit{On monoids of endomorphisms of a cycle graph},
Mathematica Slovaca 74(5), (2024), 1071--1088.

\bibitem{East:2011}
J. East,
Generators and relations for partition monoids and algebras,
J. Algebra 339 (2011), 1--26.

\bibitem{Fan:1996}
S. Fan,
\textit{On End-regular graphs},
Discrete Mathematics 159 (1996), 95--102.

\bibitem{Fernandes:2001}
V.H. Fernandes,
The monoid of all injective order preserving partial transformations on a finite chain,
Semigroup Forum 62 (2001), 178-204.

\bibitem{Fernandes:2002survey}
V.H. Fernandes,
\textit{Presentations for some monoids of partial transformations on a finite chain: a survey,
Semigroups, Algorithms, Automata and Languages},
eds. Gracinda M. S. Gomes \& Jean-\'Eric Pin \& Pedro V. Silva,
World Scientific (2002), 363--378.

\bibitem{Fernandes&Gomes&Jesus:2004}
V.H. Fernandes, G.M.S. Gomes and M.M. Jesus,
Presentations for some monoids of injective partial transformations on a finite chain,
Southeast Asian Bull. Math. 28 (2004), 903--918.

\bibitem{Fernandes&Paulista:2023}
V.H. Fernandes and T. Paulista,
\textit{On the monoid of partial isometries of a finite star graph},
Commun. Algebra 51 (2023), 1028--1048.

\bibitem{Fernandes&Quinteiro:2016}
V.H. Fernandes and T.M. Quinteiro,
Presentations for monoids of finite partial isometries,
Semigroup Forum 93 (2016), 97--110.

\bibitem{GAP4}
  The GAP~Group, \emph{GAP -- Groups, Algorithms, and Programming,
  Version 4.13.0}; 2024. \newline (https://www.gap-system.org)

\bibitem{Gu&Hou:2016}
R. Gu and H. Hou,
\textit{End-regular and End-orthodox generalized lexicographic products of bipartite graphs},
Open Math. 14 (2016), 229--236.

\bibitem{Hou&Luo&Fan:2012}
H. Hou, Y. Luo and S. Fan,
\textit{End-regular and End-orthodox joins of split graphs},
Ars Combinatoria 105 (2012), 305--318.

\bibitem{Howie:1995}
J.M. Howie,
\textit{Fundamentals of Semigroup Theory},
Clarendon Press, Oxford, 1995.

\bibitem{Howie&Ruskuc:1995}
J.M. Howie and N. Ru\v{s}kuc,
Constructions and presentations for monoids,
Commun. Algebra 22 (1994), 6209--6224.

\bibitem{Kelarev:2003}
A.V. Kelarev,
\textit{Graph Algebras and Automata}, Marcel Dekker, New York, 2003.

\bibitem{Knauer:2011}
U. Knauer,
\textit{Algebraic graph theory: morphisms, monoids, and matrices},
De Gruyter, Berlin, 2011.

\bibitem{Knauer&Wanichsombat:2014}
U. Knauer and A. Wanichsombat,
\textit{Completely Regular Endomorphisms of Split Graphs},
Ars Combinatoria 115 (2014), 357--366.

\bibitem{Lallement:1979}
G. Lallement,
Semigroups and Combinatorial Applications,
John Wiley \& Sons, New York, 1979.

\bibitem{Li:2003}
W. Li,
\textit{Graphs with regular monoids},
Discrete Mathematics 265 (2003), 105--118.

\bibitem{Moore:1897}
E.H. Moore,
Concerning the abstract groups of order $k!$ and
	$\frac12k!$ holohedrically isomorphic with the symmetric and
	the alternating substitution groups on $k$ letters,
Proc. London Math. Soc. 28 (1897), 357--366.

\bibitem{Popova:1961}
L.M. Popova,
The defining relations of certain semigroups of partial transformations of a finite set,
Leningrad. Gos. Ped. Inst. U\v cen. Zap. 218 (1961), 191--212 (Russian).

\bibitem{Popova:1962}
L.M. Popova,
Defining relations of a semigroup of partial endomorphisms of a finite linearly ordered set,
Leningrad. Gos. Ped. Inst. U\v cen. Zap. 238 (1962), 78--88 (Russian).

\bibitem{Ruskuc:1995}
N. Ru\v{s}kuc,
Semigroup Presentations,
Ph. D. Thesis, University of St-Andrews, 1995.

\bibitem{Wilkeit:1996}
E. Wilkeit,
\textit{Graphs with a regular endomorphism monoid},
Arch. Math. 66 (1996), 344--352.

\end{thebibliography}
\end{document}